\documentclass[12pt,a4paper]{article}
\usepackage{amsmath}
\usepackage{amsthm}
\usepackage{amsfonts}
\usepackage{amssymb}
\usepackage{stmaryrd}
\usepackage{latexsym}

\newtheorem{theorem}{Theorem}
\newtheorem{lemma}{Lemma}
\newtheorem{corollary}{Corollary}

\theoremstyle{definition}
\newtheorem{definition}{Definition}

\newtheorem{example}{Example} 

\begin{document}

 \tolerance2500

\title{\Large{\textbf{Schr\" oder T-quasigroups}}}
\author{\normalsize {V.A.~Shcherbacov}
}

 \maketitle

\begin{abstract}
  We prolong research of Schr\" oder quasigroups and Schr\" oder T-quasigroups.
  \medskip


\medskip

\end{abstract}

\bigskip

\section{Introduction}

 This paper  is connected in main with the following {\bf Belousov problem 1 \cite{BELPAR_05}} : Find a complete characterization of groups isotopic to quasigroups which satisfy one of the identities:  $x(y \cdot yx) = y$, $ xy \cdot yx = y$ (Stein 3rd law \cite{AS_57}), $ xy \cdot yx = x$ (Schr\" oder 2nd law \cite{AS_57}).

Some of these identities guarantee that a quasigroup is orthogonal to its parastrophe \cite{BELPAR_05, 2017_Scerb}.

Here we find conditions  imposed  not only on  a group $G$ over which various Schr\" oder  quasigroups are defined, but on the group of automorphisms $Aut \, G$, over which various T-objects (quasigroups and groupoids)  are defined.

Necessary definitions can be found in \cite{VD, RHB,  HOP, 2017_Scerb}.

\begin{definition} \label{MAIN_QUAS_DEF} Binary groupoid $(Q, \circ)$ is called a left quasigroup if for any ordered pair $(a, b)\in Q^2$
there exist the unique solution $x \in Q$ to the equation $a \circ x = b$ \cite{VD}.
\end{definition}

 \begin{definition} \label{MAIN_QUAS_DEF1}
 Binary groupoid $(Q, \circ)$ is called a right quasigroup if for any ordered pair $(a, b)\in Q^2$
there exist the unique solution $y \in Q$ to the equation $y \circ a = b$ \cite{VD}.
\end{definition}

\begin{definition}
A quasigroup $(Q,\cdot)$ with an   element $1 \in Q$, such that $1\cdot x = x\cdot 1 = x$ for all $x\in Q$, is called a {\it loop}. \label{Loop_Existenc}
\end{definition}

 \begin{definition} \label{MEDIAL_DEF}
 Binary groupoid $(Q, \cdot)$ is called medial  if this groupoid satisfies the following  medial identity:
  \begin{equation} \label{medial}
  xy\cdot uv = xu\cdot yv
  \end{equation}
  for all  $x, y, u, v \in Q$  \cite{VD}.
\end{definition}

We recall

\begin{definition}  Quasigroup $(Q, \cdot)$  is a T-quasigroup if and only if there exists an abelian group $(Q, +)$,
its automorphisms $\varphi$  and $\psi$  and a fixed element $a\in  Q$  such that $x\cdot y = \varphi x + \psi y + a$
  for all $x,  y \in  Q$ \cite{tkpn}.
\end{definition}

 A T-quasigroup with the additional condition $\varphi \psi = \psi \varphi$  is  medial.

\begin{definition} \label{Def-6}
Garrett Birkhoff \cite{BIRKHOFF} has defined an equational quasigroup as an algebra with
three binary operations $(Q, \cdot, \slash, \backslash)$ that satisfies the following six identities:
\begin{equation}
x \cdot(x\backslash  y) =y,
\end{equation}

\begin{equation}
(y \slash x)\cdot x = y,
\end{equation}

\begin{equation}
x\backslash (x \cdot y) = y,
\end{equation}

\begin{equation}
(y \cdot x) \slash x = y,
\end{equation}

\begin{equation}
x\slash (y  \backslash x) =  y,
\end{equation}

\begin{equation}
(x\slash y)\backslash x = y.
\end{equation}
\end{definition}

Ernst Schr\" oder (a German mathematician mainly known for his work on algebraic logic) introduced and studied the following identity of generalized associativity \cite{ESchr}:
\begin{equation} \label{classicaL}
                                (y\cdot z) \backslash x =  z (x\cdot y).                                                                                                                                 \end{equation}

See also \cite{Ibragimov_1, IBRAGIMOV,  KUZ_AV} for detail. Notice, article \cite{KUZ_AV} is well written and very detailed.

In the quasigroup case the identity (\ref{classicaL}) is equivalent to the following identity \cite{{PUSH_14, 2017_Scerb}}:
\begin{equation}  \label{SCH_1}
                                                                      (y\cdot z) \cdot(z\cdot(x\cdot y)) = x                                                 \end{equation}

If in idempotent quasigroup $(Q;\cdot)$ the identity (\ref{SCH_1}) we put $x =  y$, then we obtain the following
standard  Schr\" oder identity:
\begin{equation} \label{2_61}
                                (x\cdot y)\cdot (y\cdot x) = x.
\end{equation}
\begin{definition}
Any quasigroup with the identity (\ref{2_61}) is called a Schr\" oder quasigroup.
\end{definition}

So we have  different objects that have name Schr\" oder.  Namely,

(i)
the following identity of generalized associativity \cite{ESchr}:

                                $(y\cdot z) \backslash x =  z (x\cdot y)$  (\ref{classicaL});

(ii) the Schr\" oder identity of generalized associativity in quasigroups: (\ref{SCH_1});

 (iii)   Schr\" oder identity (Schr\" oder 2-nd identity \cite{AS_57})  $(x\cdot y)\cdot (y\cdot x) = x$ (\ref{2_61});

Identity $(x\cdot y)\cdot (y\cdot x) = y$ is named by Albert Sade \cite{AS_57} as  Stein 3-rd identity.  \label{Belo-4}

Each of these  identities deserves a separate study in the class
of groupoids, left (right) quasigroups, in the classes of quasigroups and of T-quasigroups.

\subsection{ Schr\" oder identity of generalized associativity}

It is convenient to call this identity the Schr\" oder identity of generalized associativity. Often various variants of associative identity, which are true in a quasigroup, guarantee that this quasigroup is a loop.

It is not so in the case with the identity.  We give an example of quasigroup which is not a loop  with the identity (\ref{SCH_1}) \cite{PUSH_14}. See also \cite{2017_Scerb}. A quasigroup from  this example does not have left and right identity element.

Quasigroups with Schr\" oder identity of generalized associativity are not necessary idempotent and associative. See  the following example.

 \centerline{Table 1}
  $$
\begin{array}{l|llllllll}
\cdot &0	&1	&2	&3	&4 &5 &6 &7  \\
\hline
0 &1 &4 &7 &0 &6 &5 &2 &3 \\
1 &5 &2 &3 &6 &0 &1 &4 &7 \\
2 &0 &7 &4 &1 &5 &6 &3 &2 \\
3 &6 &3 &2 &5 &1 &0 &7 &4 \\
4 &4 &1 &0 &7 &3 &2 &5 &6  \\
5 &3 &6 &5 &2 &4 &7 &0 &1 \\
6 &7 &0 &1 &4 &2 &3 &6 &5 \\
7 &2 &5 &6 &3 &7 &4 &1 &0 \\
\end{array}
$$

The left cancellation (left division) groupoid with the identity (\ref{SCH_1}) and with the identity $(x\slash x=  y\slash y)$ (in a quasigroup this identity guarantees existence of the left identity element) is a commutative group of exponent  two \cite{PUSH_14}.

The similar results are true for the right case \cite{PUSH_14}.
In this case we use the identity $(x\backslash x=  y\backslash y)$.

It is clear that this result is true for any quasigroup with the left or right  identity element.

Notice, any 2-group $(G, +)$ (in such group $x+x=0$ for any $x\in G$) satisfies Schr\" oder identity of generalized associativity.

\subsection{ Schr\" oder identity  $(x\cdot y)\cdot (y\cdot x) = x$}

We recall information from \cite{CCL_80}.

\begin{lemma}
 In any Schr\" oder quasigroup $(Q, \cdot)$  the equality $x \cdot x = y \cdot y$ implies $x = y$, and the equality $x \cdot y = y \cdot x$ implies $x = y$.
\end{lemma}
\begin{proof} Suppose $(x \cdot x) = (y \cdot y)$.
Then from the identity (\ref{2_61}) we have $x  = (x \cdot x) \cdot (x \cdot x) =  (y \cdot y) \cdot (y \cdot y) = y.$

Suppose $x \cdot y = y \cdot x$. Then we have $y = (x \cdot y)\cdot (y\cdot  x)  = (y \cdot x) \cdot  (x \cdot y) =  x$.
\end{proof}

\begin{theorem}  \label{TheoremL_80}
  Necessary condition for the existence of an idempotent Schr\" oder quasigroup $(Q,  \cdot)$ of order v is that v = 0 or = 1  pmod 4 \cite{CCL_80}.
\end{theorem}

In the proofs of Theorems \ref{TheoremL_80}   are used sufficiently results from \cite{WILSON_Rich}.

\begin{example} Define groupoid $GF(2^r, \ast) $  over the Galois field $GF( 2^r)$ in the following way:
\begin{equation}
x \ast y = a \cdot x +(a+1) y,
 \end{equation}
  where $x, y \in  GF(2^r)$, the element $a$ is a fixed element of the set $(GF(2^r))$, the operations $+$ and  $\cdot$ are binary operations of this field.

The groupoid $(GF(2^r), \ast)$ is an idempotent medial Schr\" oder quasigroup \cite{CCL_80}.
\end{example}

\begin{example}
In \cite{CCL_80} it is mentioned that do not exist Schr\" oder quasigroups $(Q, \cdot)$ of order 5.
Using Mace 5 we construct  quasigroups with Stein 3-rd law   of order 5 (left loop  and idempotent quasigroup).

\begin{center}
\begin{tabular}{r|rrrrr}
$\star$ & 0 & 1 & 2 & 3 & 4\\
\hline
    0 & 0 & 1 & 2 & 3 & 4 \\
    1 & 2 & 3 & 1 & 4 & 0 \\
    2 & 3 & 0 & 4 & 2 & 1 \\
    3 & 4 & 2 & 0 & 1 & 3 \\
    4 & 1 & 4 & 3 & 0 & 2
\end{tabular}
\begin{tabular}{r|rrrrr}
$\cdot$ & 0 & 1 & 2 & 3 & 4\\
\hline
    0 & 0 & 2 & 3 & 4 & 1 \\
    1 & 3 & 1 & 4 & 2 & 0 \\
    2 & 4 & 0 & 2 & 1 & 3 \\
    3 & 1 & 4 & 0 & 3 & 2 \\
    4 & 2 & 3 & 1 & 0 & 4
\end{tabular}
\end{center}
\end{example}

We remark, in quasigroups from Schr\" oder identity does  not follow Stein 3-rd  (quasigroup $(Q, \ast)$) and vice.

In quasigroup $(Q, \ast)$ is true identity (\ref{2_61}) and is not true identity (\ref{Belo-4}),
in quasigroup $(Q, \circ)$ is true identity (\ref{Belo-4})  and is not true identity (\ref{2_61}).

\begin{center}
\begin{tabular}{r|rrrr}
$\ast$ & 0 & 1 & 2 & 3\\
\hline
    0 & 0 & 1 & 3 & 2 \\
    1 & 2 & 3 & 1 & 0 \\
    2 & 1 & 0 & 2 & 3 \\
    3 & 3 & 2 & 0 & 1
\end{tabular} \hspace{.5cm}
\begin{tabular}{r|rrrr}
$\circ$ & 0 & 1 & 2 & 3\\
\hline
    0 & 0 & 2 & 3 & 1 \\
    1 & 3 & 1 & 0 & 2 \\
    2 & 1 & 3 & 2 & 0 \\
    3 & 2 & 0 & 1 & 3
\end{tabular}
\end{center}

\subsection{Quasigroups with identity $x(y \cdot yx) = y$}

Next Cayley tables of quasigroups with identity $x(y \cdot yx) = y$ were obtained using Mace \cite{MAC_CUNE_MACE}.

\begin{tabular}{r|rrr}
*: & 0 & 1 & 2\\
\hline
    0 & 0 & 1 & 2 \\
    1 & 2 & 0 & 1 \\
    2 & 1 & 2 & 0
\end{tabular}
\begin{tabular}{r|rrrr}
*: & 0 & 1 & 2 & 3\\
\hline
    0 & 0 & 2 & 3 & 1 \\
    1 & 1 & 3 & 2 & 0 \\
    2 & 2 & 0 & 1 & 3 \\
    3 & 3 & 1 & 0 & 2
\end{tabular}
\begin{tabular}{r|rrrrr}
*: & 0 & 1 & 2 & 3 & 4\\
\hline
    0 & 0 & 2 & 4 & 1 & 3 \\
    1 & 2 & 1 & 3 & 4 & 0 \\
    2 & 4 & 3 & 2 & 0 & 1 \\
    3 & 1 & 4 & 0 & 3 & 2 \\
    4 & 3 & 0 & 1 & 2 & 4
\end{tabular}
\begin{tabular}{r|rrrrrrr}
*: & 0 & 1 & 2 & 3 & 4 & 5 & 6\\
\hline
    0 & 0 & 2 & 3 & 1 & 5 & 6 & 4 \\
    1 & 4 & 1 & 6 & 0 & 3 & 2 & 5 \\
    2 & 5 & 0 & 2 & 4 & 6 & 1 & 3 \\
    3 & 6 & 5 & 0 & 3 & 1 & 4 & 2 \\
    4 & 1 & 6 & 5 & 2 & 4 & 3 & 0 \\
    5 & 2 & 3 & 4 & 6 & 0 & 5 & 1 \\
    6 & 3 & 4 & 1 & 5 & 2 & 0 & 6
\end{tabular}

\begin{tabular}{r|rrrrrrrr}
*: & 0 & 1 & 2 & 3 & 4 & 5 & 6 & 7\\
\hline
    0 & 0 & 2 & 4 & 1 & 6 & 3 & 7 & 5 \\
    1 & 6 & 1 & 5 & 2 & 0 & 7 & 3 & 4 \\
    2 & 7 & 4 & 2 & 5 & 3 & 6 & 0 & 1 \\
    3 & 4 & 7 & 0 & 3 & 5 & 1 & 2 & 6 \\
    4 & 5 & 3 & 6 & 7 & 4 & 2 & 1 & 0 \\
    5 & 2 & 0 & 7 & 6 & 1 & 5 & 4 & 3 \\
    6 & 3 & 5 & 1 & 4 & 7 & 0 & 6 & 2 \\
    7 & 1 & 6 & 3 & 0 & 2 & 4 & 5 & 7
\end{tabular}
\begin{tabular}{r|rrrrrrrrr}
*: & 0 & 1 & 2 & 3 & 4 & 5 & 6 & 7 & 8\\
\hline
    0 & 0 & 1 & 2 & 3 & 4 & 5 & 6 & 7 & 8 \\
    1 & 2 & 0 & 1 & 8 & 6 & 4 & 5 & 3 & 7 \\
    2 & 1 & 2 & 0 & 7 & 5 & 6 & 4 & 8 & 3 \\
    3 & 4 & 5 & 6 & 0 & 3 & 7 & 8 & 1 & 2 \\
    4 & 3 & 7 & 8 & 4 & 0 & 1 & 2 & 5 & 6 \\
    5 & 8 & 3 & 7 & 6 & 2 & 0 & 1 & 4 & 5 \\
    6 & 7 & 8 & 3 & 5 & 1 & 2 & 0 & 6 & 4 \\
    7 & 6 & 4 & 5 & 2 & 8 & 3 & 7 & 0 & 1 \\
    8 & 5 & 6 & 4 & 1 & 7 & 8 & 3 & 2 & 0
\end{tabular}

\begin{tabular}{r|rrrrrrrrrrr}
*: & 0 & 1 & 2 & 3 & 4 & 5 & 6 & 7 & 8 & 9 & 10\\
\hline
    0 & 0 & 2 & 4 & 1 & 5 & 3 & 7 & 9 & 6 & 10 & 8 \\
    1 & 8 & 1 & 7 & 10 & 0 & 2 & 9 & 6 & 3 & 5 & 4 \\
    2 & 6 & 8 & 2 & 4 & 9 & 0 & 1 & 10 & 5 & 7 & 3 \\
    3 & 10 & 6 & 0 & 3 & 1 & 9 & 8 & 4 & 7 & 2 & 5 \\
    4 & 7 & 5 & 6 & 0 & 4 & 10 & 3 & 2 & 1 & 8 & 9 \\
    5 & 9 & 0 & 3 & 8 & 7 & 5 & 2 & 1 & 10 & 4 & 6 \\
    6 & 4 & 10 & 9 & 5 & 8 & 1 & 6 & 3 & 2 & 0 & 7 \\
    7 & 5 & 3 & 8 & 2 & 10 & 6 & 4 & 7 & 9 & 1 & 0 \\
    8 & 2 & 7 & 10 & 9 & 3 & 4 & 5 & 0 & 8 & 6 & 1 \\
    9 & 3 & 4 & 1 & 7 & 6 & 8 & 10 & 5 & 0 & 9 & 2 \\
    10 & 1 & 9 & 5 & 6 & 2 & 7 & 0 & 8 & 4 & 3 & 10
\end{tabular}

\begin{tabular}{r|rrrrrrrrrrrr}
*: & 0 & 1 & 2 & 3 & 4 & 5 & 6 & 7 & 8 & 9 & 10 & 11\\
\hline
    0 & 1 & 0 & 2 & 3 & 4 & 7 & 6 & 5 & 8 & 11 & 10 & 9 \\
    1 & 0 & 1 & 3 & 2 & 5 & 4 & 7 & 6 & 9 & 8 & 11 & 10 \\
    2 & 2 & 3 & 0 & 1 & 6 & 5 & 4 & 7 & 10 & 9 & 8 & 11 \\
    3 & 3 & 2 & 1 & 0 & 7 & 6 & 5 & 4 & 11 & 10 & 9 & 8 \\
    4 & 4 & 5 & 6 & 7 & 8 & 9 & 10 & 11 & 0 & 1 & 2 & 3 \\
    5 & 7 & 4 & 5 & 6 & 10 & 8 & 11 & 9 & 2 & 0 & 3 & 1 \\
    6 & 6 & 7 & 4 & 5 & 9 & 11 & 8 & 10 & 1 & 3 & 0 & 2 \\
    7 & 5 & 6 & 7 & 4 & 11 & 10 & 9 & 8 & 3 & 2 & 1 & 0 \\
    8 & 8 & 9 & 10 & 11 & 0 & 2 & 1 & 3 & 4 & 5 & 6 & 7 \\
    9 & 11 & 8 & 9 & 10 & 1 & 0 & 3 & 2 & 5 & 6 & 7 & 4 \\
    10 & 10 & 11 & 8 & 9 & 2 & 3 & 0 & 1 & 6 & 7 & 4 & 5 \\
    11 & 9 & 10 & 11 & 8 & 3 & 1 & 2 & 0 & 7 & 4 & 5 & 6
\end{tabular}

We recall,
\begin{lemma} \label{direct-prod}
If two algebraic systems, say $A$ and $B$ satisfy  an identity, then direct product of these systems $A\times B$ satisfy this identity \cite{BURRIS, 2017_Scerb}.
\end{lemma}
\begin{proof}
It follows from Definition \ref{Def-6}.
\end{proof}

\begin{lemma}
Quasigroups with identity $x(y \cdot yx) = y$ there exists for any $n$, where
\begin{equation}
n \in \{ 3^{n_1} \cdot 4^{n_2}\cdot 5^{n_3} \cdot 7^{n_4}\cdot 8^{n_5}\cdot 11^{n_6} \cdot 23^{n_7}\},
\end{equation}
 $n_i \in \mathbb{N}$, $i\in \overline{\{1,7\}}$.
\end{lemma}
\begin{proof}
This follows from Lemma \ref{direct-prod},  given in this section examples and Example \ref{ex-5}.
\end{proof}

\subsection{Stein 3-rd  identity}

Next Cayley tables of quasigroups with identity $xy \cdot yx = y$ were obtained using Mace \cite{MAC_CUNE_MACE}.

\begin{tabular}{r|rrrr}
*: & 0 & 1 & 2 & 3\\
\hline
    0 & 0 & 1 & 3 & 2 \\
    1 & 2 & 3 & 1 & 0 \\
    2 & 1 & 0 & 2 & 3 \\
    3 & 3 & 2 & 0 & 1
\end{tabular}
\begin{tabular}{r|rrrrr}
*: & 0 & 1 & 2 & 3 & 4\\
\hline
    0 & 0 & 1 & 2 & 3 & 4 \\
    1 & 2 & 3 & 1 & 4 & 0 \\
    2 & 3 & 0 & 4 & 2 & 1 \\
    3 & 4 & 2 & 0 & 1 & 3 \\
    4 & 1 & 4 & 3 & 0 & 2
\end{tabular}
\begin{tabular}{r|rrrrrr}
*: & 0 & 1 & 2 & 3 & 4 & 5\\
\hline
    0 & 1 & 0 & 2 & 3 & 4 & 5 \\
    1 & 0 & 1 & 3 & 2 & 5 & 4 \\
    2 & 2 & 3 & 4 & 5 & 0 & 1 \\
    3 & 3 & 2 & 5 & 4 & 1 & 0 \\
    4 & 4 & 5 & 0 & 1 & 2 & 3 \\
    5 & 5 & 4 & 1 & 0 & 3 & 2
\end{tabular}

\begin{tabular}{r|rrrrrrr}
*: & 0 & 1 & 2 & 3 & 4 & 5 & 6\\
\hline
    0 & 1 & 0 & 2 & 3 & 4 & 6 & 5 \\
    1 & 0 & 1 & 3 & 2 & 5 & 4 & 6 \\
    2 & 2 & 3 & 0 & 6 & 1 & 5 & 4 \\
    3 & 3 & 2 & 4 & 5 & 6 & 0 & 1 \\
    4 & 5 & 6 & 1 & 4 & 0 & 2 & 3 \\
    5 & 4 & 5 & 6 & 1 & 2 & 3 & 0 \\
    6 & 6 & 4 & 5 & 0 & 3 & 1 & 2
\end{tabular}
\begin{tabular}{r|rrrrrrrr}
*: & 0 & 1 & 2 & 3 & 4 & 5 & 6 & 7\\
\hline
    0 & 1 & 0 & 2 & 3 & 4 & 7 & 6 & 5 \\
    1 & 0 & 1 & 3 & 2 & 5 & 4 & 7 & 6 \\
    2 & 2 & 3 & 0 & 1 & 6 & 5 & 4 & 7 \\
    3 & 3 & 2 & 1 & 0 & 7 & 6 & 5 & 4 \\
    4 & 4 & 5 & 6 & 7 & 0 & 1 & 2 & 3 \\
    5 & 7 & 4 & 5 & 6 & 1 & 2 & 3 & 0 \\
    6 & 6 & 7 & 4 & 5 & 2 & 3 & 0 & 1 \\
    7 & 5 & 6 & 7 & 4 & 3 & 0 & 1 & 2
\end{tabular}

\begin{tabular}{r|rrrrrrrrr}
*: & 0 & 1 & 2 & 3 & 4 & 5 & 6 & 7 & 8\\
\hline
    0 & 1 & 0 & 2 & 3 & 4 & 8 & 7 & 6 & 5 \\
    1 & 0 & 1 & 3 & 2 & 5 & 4 & 8 & 7 & 6 \\
    2 & 2 & 3 & 0 & 1 & 6 & 5 & 4 & 8 & 7 \\
    3 & 3 & 2 & 1 & 0 & 7 & 6 & 5 & 4 & 8 \\
    4 & 4 & 5 & 6 & 7 & 8 & 1 & 0 & 3 & 2 \\
    5 & 8 & 4 & 5 & 6 & 0 & 7 & 3 & 2 & 1 \\
    6 & 7 & 8 & 4 & 5 & 1 & 2 & 6 & 0 & 3 \\
    7 & 6 & 7 & 8 & 4 & 2 & 3 & 1 & 5 & 0 \\
    8 & 5 & 6 & 7 & 8 & 3 & 0 & 2 & 1 & 4
\end{tabular}
\begin{tabular}{r|rrrrrrrrrr}
*: & 0 & 1 & 2 & 3 & 4 & 5 & 6 & 7 & 8 & 9\\
\hline
    0 & 1 & 0 & 2 & 3 & 4 & 7 & 8 & 5 & 6 & 9 \\
    1 & 0 & 1 & 3 & 2 & 5 & 4 & 9 & 8 & 7 & 6 \\
    2 & 2 & 3 & 0 & 1 & 6 & 5 & 4 & 9 & 8 & 7 \\
    3 & 3 & 2 & 1 & 0 & 7 & 6 & 5 & 4 & 9 & 8 \\
    4 & 4 & 5 & 6 & 7 & 8 & 9 & 0 & 1 & 2 & 3 \\
    5 & 8 & 7 & 9 & 6 & 0 & 1 & 2 & 3 & 4 & 5 \\
    6 & 9 & 8 & 7 & 5 & 1 & 2 & 3 & 6 & 0 & 4 \\
    7 & 5 & 4 & 8 & 9 & 2 & 3 & 6 & 7 & 1 & 0 \\
    8 & 6 & 9 & 4 & 8 & 3 & 0 & 7 & 2 & 5 & 1 \\
    9 & 7 & 6 & 5 & 4 & 9 & 8 & 1 & 0 & 3 & 2
\end{tabular}

\begin{lemma}
Quasigroups with identity $xy \cdot yx = y$ there exists for any $n$, where
\begin{equation}
n \in \{ 4^{n_1} \cdot 5^{n_2}\cdot 6^{n_3} \cdot 7^{n_4}\cdot 8^{n_5}\cdot 9^{n_6} \cdot 10^{n_7}\cdot 13^{n_8}\cdot 17^{n_9}\cdot 29^{n_{10}}\},
\end{equation}
 $n_i \in \mathbb{N}$, $i\in \overline{\{1,10\}}$.
\end{lemma}
\begin{proof}
This follows from Lemma \ref{direct-prod}, given in this paragraph  examples and Examples \ref{example7}, \ref{example8}, \ref{example9}.
\end{proof}

\section{Schr\" oder identities   in T-quasigroups}

\subsection{Schr\" oder identity of generalized associativity  in T-quasigroups}

\begin{theorem}
In T-quasigroup $(Q, \cdot)$ of the form $x\cdot y = \varphi x + \psi y$ Schr\" oder identity of generalized associativity is true if and only if $\varphi x = \psi^{-2} x$, $\varepsilon = \varphi^7$, $\varepsilon =  \psi^{14}$,  $\varphi \psi z + \psi \varphi  z = 0$.
\end{theorem}
\begin{proof}
We rewrite identity (\ref{SCH_1}) in the following  form:
\begin{equation} \label{Osnovn}
\varphi^2 y + \psi^3 y + \varphi\psi z + \psi\varphi z +
\psi^2 \varphi x = x.
\end{equation}
If we substitute in equality (\ref{Osnovn}) $y=z=0$ then we have
\begin{equation} \label{111}
\varphi x = \psi^{-2} x.
\end{equation}
If we substitute in equality (\ref{Osnovn}) $x=z=0$ then we have
\begin{equation} \label{222}
\varphi^2 y + \psi^3 y = 0.
\end{equation}
Taking into consideration equality (\ref{111}), we can re-write  equality (\ref{222}) in the form
\begin{equation} \label{2222}
\psi^{-4} y + \psi^3 y = 0,
\end{equation}
or in the form
\begin{equation} \label{22222}
\psi^3 = I\psi^{-4},
\end{equation}
where $Ix= -x$ for all $x\in Q$. Notice, the permutation  $I$ is an automorphism of the group $(Q, +)$ here.
Therefore, we can rewrite previous equalities in the form
\begin{equation} \label{222222}
\varepsilon  = I\psi^{-7}, I  = \psi^{-7}, \varepsilon = \psi^{-14}, \varepsilon =  \psi^{14}, \varepsilon = \varphi^7.
\end{equation}

If we substitute in equality (\ref{Osnovn}) $x=y=0$ then we have
\begin{equation} \label{333}
\varphi\psi z + \psi\varphi z = 0.
\end{equation}

Converse. If we substitute in identity (\ref{SCH_1}) the expression $x\cdot y = \varphi x + \psi y$, then we obtain equality (\ref{Osnovn}), which is true taking into consideration equalities (\ref{111}), (\ref{222}), (\ref{333}). Then we obtain, that identity (\ref{SCH_1}) is true in this case.
\end{proof}

\begin{corollary}
In medial quasigroup $(Q, \cdot)$ of the form $x\cdot y = \varphi x + \psi y$ Schr\" oder identity of generalized associativity is true if and only if the group  $(Q, +)$ is an abelian 2-group (i.e. $x+x=0$ for any $x\in Q$),  $\varphi x = \psi^{-2} x$, $\varepsilon = \varphi^7$, $\varepsilon =  \psi^{14}$.
\end{corollary}
\begin{proof}
From the identity of mediality it follows that $\varphi \psi z + \psi \varphi  z = 2\cdot \varphi \psi z = 0$ for all $z\in Q$, i.e., the group  $(Q, +)$ is an abelian 2-group.
\end{proof}

\begin{example}
We start from the group $ GL(3,2)\cong PSL(2, 7)$. This is the group of non-degenerate matrices of size  $3\times 3$  over the field of order 2. Or the group of non-degenerate matrices of size  $2\times 2$  over the field of order 7. The order of this group is equal to 8, $|Aut(GL(3, 2)| = 168= 3\times 7\times 8$.

We can present elements of the group $ (Z_2^3, +)$ in the following  form : $1 =(000), 2=(001), 3=(010), 4=(011), 5 = (100), 6=(101), 7=(110), 8 = (111).$

 $$
\begin{array}{l|llllllll}
+ &1	&2	&3	&4	&5 &6 &7 &8  \\
\hline
1 &1 &2 &3 &4 &5 &6 &7 &8 \\
2 &2 &1 &4 &3 &6 &5 &8 &7 \\
3 &3 &4 &1 &2 &7 &8 &5 &6 \\
4 &4 &3 &1 &1 &8 &7 &6 &5 \\
5 &5 &6 &7 &8 &1 &2 &3 &4  \\
6 &6 &5 &8 &7 &2 &1 &4 &3 \\
7 &7 &8 &5 &6 &3 &4 &1 &2 \\
8 &8 &7 &6 &5 &4 &3 &2 &1 \\
\end{array}
$$

We have the following automorphisms of the group $ (Z_2^3, +)$:

\smallskip

$\varphi =\begin{pmatrix}
1 & 1 & 0\\
1&  0 & 1\\
0& 1& 0 \\
\end{pmatrix}$,
$\psi =\begin{pmatrix}
0& 1 & 0 \\
1& 1 & 1 \\
0& 1 & 1 \\
\end{pmatrix}$.

Notice that  $\varphi^7=\psi^7= \varepsilon$, $\varphi=\psi^{-2}$, $\varphi \psi = \psi \varphi$.

Therefore Schr\" oder medial quasigroup $(Q, \circ)$ of generalized associativity can have the form $x\circ y = \varphi x + \psi y$:

 $$
\begin{array}{l|llllllll}
\circ &1	&2	&3	&4	&5 &6 &7 &8  \\
\hline
1 &1 &4 &8 &5 &3 &2 &6 &7 \\
2 &3 &2 &6 &7 &1 &4 &8 &5 \\
3 &6 &7 &3 &2 &8 &5 &1 &4 \\
4 &8 &5 &1 &4 &6 &7 &3 &2 \\
5 &7 &6 &2 &3 &5 &8 &4 &1  \\
6 &5 &8 &4 &1 &7 &6 &2 &3 \\
7 &4 &1 &5 &8 &2 &3 &7 &6 \\
8 &2 &3 &7 &6 &4 &1 &5 &8 \\
\end{array}
$$

\end{example}

\subsection{Schr\" oder identity}

\begin{theorem}
In T-quasigroup $(Q, \cdot)$ of the form $x\cdot y = \varphi x + \psi y$ Schr\" oder identity is true if and only if $\varphi^2+\psi^2= \varepsilon$,  $\varphi \psi y + \psi\varphi y = 0$.
\end{theorem}
\begin{proof}
From  identity
\begin{equation}
xy\cdot yx = x
\end{equation}
we have
\begin{equation} \label{rt_1}
\varphi(\varphi x + \psi y)+\psi (\varphi y +\psi x) = x.
\end{equation}

If in (\ref{rt_1}) $y=0$,  then
\begin{equation} \label{15U}
\varphi^2+\psi^2= \varepsilon.
\end{equation}

 If in (\ref{rt_1}) $x=0$,  then
 \begin{equation} \label{16O}
 \varphi \psi y + \psi\varphi y = 0.
\end{equation}

Converse. If we substitute in identity (\ref{2_61}) the expression $x\cdot y = \varphi x + \psi y$, then we obtain equality (\ref{rt_1}), which is true taking into consideration equalities (\ref{15U}),  (\ref{16O}). Then we obtain, that identity (\ref{2_61}) is true in this case.
\end{proof}

\begin{corollary}
In medial quasigroup $(Q, \cdot)$ of the form $x\cdot y = \varphi x + \psi y$ Schr\" oder identity is true if and only if $\varphi^2+\psi^2= \varepsilon$,  the group  $(Q, +)$ is an abelian 2-group (i.e., $x+x=0$ for any $x\in Q$).
\end{corollary}
\begin{proof}
This follows from mediality of quasigroup $(Q, \cdot)$.
\end{proof}

\subsection{T-quasigroup with identity $x(y \cdot yx) = y$}

\begin{theorem}
In T-quasigroup $(Q, \cdot)$ of the form $x\cdot y = \varphi x + \psi y$  identity $x(y \cdot yx) = y$  is true if and only if $\varphi =I\psi^3$,  $\psi^4 + \psi^5   = I$.
\end{theorem}
\begin{proof}
From  identity
\begin{equation} \label{SROd11}
x(y \cdot yx) = y
\end{equation}
we have
\begin{equation} \label{rt_21}
\varphi x + \psi (\varphi y + \psi (\varphi y  +\psi x) = y.
\end{equation}
\begin{equation} \label{rt_211}
\varphi x + \psi \varphi y + \psi^2 \varphi y + \psi^3 x = y.
\end{equation}

If in (\ref{rt_211}) $y=0$,  then
\begin{equation} \label{15US1}
\varphi+\psi^3= 0,
\end{equation}

\begin{equation} \label{15US11}
\varphi =I\psi^3.
\end{equation}

 If in (\ref{rt_211}) $x=0$,  then
 \begin{equation} \label{16OS1}
  \psi \varphi   + \psi^2 \varphi  = \varepsilon.
\end{equation}

Taking in account (\ref{15US11}) we have
 \begin{equation} \label{16OS11}
\psi^4 + \psi^5    = I.
\end{equation}

Converse. If we substitute in identity (\ref{SROd11}) the expression $x\cdot y = \varphi x + \psi y$, then we obtain equality (\ref{rt_21}), which is true taking into consideration equalities (\ref{15US1}),  (\ref{16OS1}). Then we obtain, that identity (\ref{SROd11})  is true in this case.
\end{proof}

\begin{corollary}
In medial quasigroup $(Q, \cdot)$ of the form $x\cdot y = \varphi x + \psi y$  identity  $x(y \cdot yx) = y$ is true if and only if $\varphi= I\psi^3$,  $\psi^4 + \psi^5    = I$.
\end{corollary}
\begin{proof}
This follows from mediality of quasigroup $(Q, \cdot)$.
\end{proof}
\begin{example} \label{ex-4}
Suppose we have the group  $Z_n$ of residues modulo $n$. If $\psi = -2$, then $\varphi = I\psi^3 = 8$,   $\psi^4 + \psi^5 = 16-32= -16 = -1 $. The last is true if $-16\equiv -1 \pmod n$, $15\equiv 0 \pmod n$.

Then quasigroup $(Z_{15}, \circ)$ of the form  $x\circ y = 8\cdot x-2\cdot y \pmod{15}$ is medial  quasigroup with identity $x(y \cdot yx) = y$.

Check.
$8x-2(8y-2(8y-2x))=y \pmod{15}$, $8x-16y +32y-8x= y \pmod{15}$, $16y =y \pmod{15}$, $y=y \pmod{15}$.
\end{example}

\begin{example} \label{ex-5}
Suppose we have the group  $Z_n$ of residues modulo $n$. If $\psi = -3$, then $\varphi = I\psi^3 = 27$,   $\psi^4 + \psi^5 = 81-243= -162 = -1 $. The last is true if $-162\equiv -1 \pmod n$, $161\equiv 0 \pmod n$.

Quasigroup $(Z_{23}, \circ)$ of the form  $x\circ y = 4\cdot x-3\cdot y \pmod{23}$ is medial  quasigroup with identity $x(y \cdot yx) = y$.

Check.
$4x-3(4y-3(4y-3x))=y \pmod{23}$, $4x-12y +36y-27x= y \pmod{23}$, $24y =y \pmod{23}$, $y=y \pmod{23}$.
\end{example}

\begin{example}
Quasigroup $(Z_{161}, \circ)$ of the form  $x\circ y = 27\cdot x-3\cdot y \pmod{161}$ is medial  quasigroup with identity $x(y \cdot yx) = y$.

Check.
$27x-3(27y-3(27y-3x))=y \pmod{161}$, $27x-81y +243y-27x= y \pmod{161}$, $162y =y \pmod{161}$, $y=y \pmod{161}$.

\end{example}

\subsection{T-quasigroups with Stein 3-rd identity}

\begin{theorem}
In T-quasigroup $(Q, \cdot)$ of the form $x\cdot y = \varphi x + \psi y$ Stein 3-rd  identity  is true if and only if $\varphi^2+\psi^2= 0$,  $\varphi \psi y + \psi\varphi y = \varepsilon$.
\end{theorem}
\begin{proof}
From  identity
\begin{equation} \label{SROd1}
xy\cdot yx = y
\end{equation}
we have
\begin{equation} \label{rt_2}
\varphi(\varphi x + \psi y)+\psi (\varphi y +\psi x) = y.
\end{equation}

If in (\ref{rt_2}) $y=0$,  then
\begin{equation} \label{15US}
\varphi^2+\psi^2= 0.
\end{equation}

\begin{equation} \label{15US1}
\varphi^2= I\psi^2.
\end{equation}

 If in (\ref{rt_2}) $x=0$,  then
 \begin{equation} \label{16OS}
 \varphi \psi  + \psi\varphi  = \varepsilon.
\end{equation}

Converse. If we substitute in identity (\ref{SROd1}) the expression $x\cdot y = \varphi x + \psi y$, then we obtain equality (\ref{rt_2}), which is true taking into consideration equalities (\ref{15US}),  (\ref{16OS}). Then we obtain, that identity (\ref{SROd1})  is true in this case.
\end{proof}

\begin{corollary}
In medial quasigroup $(Q, \cdot)$ of the form $x\cdot y = \varphi x + \psi y$ Stein 3-rd identity  is true if and only if $\varphi^2+\psi^2= 0$,  $2 \varphi \psi =\varepsilon$.
\end{corollary}
\begin{proof}
This follows from mediality of quasigroup $(Q, \cdot)$.
\end{proof}

\begin{example}
Suppose we have the group  $Z_n$ of residues modulo $n$. If $\varphi = 1$, $\psi =3$ then $\varphi^2+\psi^2 = 1+9=0 \pmod{5}, n = 5$. Further $2 \varphi \psi = 2\cdot 1\cdot 3 = 6 = \varepsilon = 1\pmod{5} $, $x\cdot y = x +3y \pmod{5}$.

Check.
$(x+3y)+ 3(y+3x)=y \pmod{5}$, $x+3y +3y +9x= y \pmod{5}$,  $y=y \pmod{5}$.
\end{example}

\begin{example}
Suppose we have the group  $Z_n$ of residues modulo $n$. If $\varphi = 4$, $\psi =2$ then $\varphi^2+\psi^2 = 16+4=0 \pmod{5}, n = 5$. Further $2 \varphi \psi = 2\cdot 4\cdot 2 = 16 = \varepsilon = 1\pmod{5} $, $x\cdot y = 4x +2y \pmod{5}$.

Check.
$4(4x+2y)+ 2(4y+2x)=y \pmod{5}$, $16x+8y +8y +4x= y \pmod{5}$,  $y=y \pmod{5}$.
\end{example}

\begin{example} \label{example7}
Suppose we have the group  $Z_n$ of residues modulo $n$. If $\varphi = 2$, $\psi =10$ then $\varphi^2+\psi^2 = 4+100=104= 0 \pmod{13}, n = 13$. Further $2 \varphi \psi = 2\cdot 2\cdot 10 = 40 = \varepsilon = 1\pmod{13} $, $x\cdot y = 2x +10y \pmod{13}$.

Check.
$2(2x+10 y)+ 10(2y+10 x)=y \pmod{13}$, $4x+20y +20y +100x= y \pmod{13}$,  $y=y \pmod{13}$.
\end{example}

\begin{example} \label{example8}
Suppose we have the group  $Z_n$ of residues modulo $n$. If $\varphi = 6$, $\psi =10$ then $\varphi^2+\psi^2 = 36+100=136= 0 \pmod{17}, n = 17$. Further $2 \varphi \psi = 2\cdot 6\cdot 10 = 120 = \varepsilon = 1\pmod{17} $, $x\cdot y = 6x +10y \pmod{17}$.

Check.
$6(6x+10 y)+ 10(6y+10 x)=y \pmod{17}$, $36x+60y +60y +100x= y \pmod{17}$,  $y=y \pmod{17}$.
\end{example}

\begin{example}\label{example9}
Suppose we have the group  $Z_n$ of residues modulo $n$. If $\varphi = 8$, $\psi =20$ then $\varphi^2+\psi^2 = 64+400=464= 0 \pmod{29}, n = 29$. Further $2 \varphi \psi = 2\cdot 8\cdot 20 = 320 = \varepsilon = 1\pmod{29} $, $x\cdot y = 8x +20y \pmod{29}$.

Check.
$8(8x+20 y)+ 20(8y+20 x)=y \pmod{29}$, $64x+160y +160y +400x= y \pmod{29}$,  $y=y \pmod{65}$.
\end{example}

\begin{example}
Suppose we have the group  $Z_n$ of residues modulo $n$. If $\varphi = 3$, $\psi =11$ then $\varphi^2+\psi^2 = 9+121=130= 0 \pmod{65}, n = 65$. Further $2 \varphi \psi = 2\cdot 3\cdot 11 = 66 = \varepsilon = 1\pmod{65} $, $x\cdot y = 3x +11y \pmod{65}$.

Check.
$3(3x+11 y)+ 11(3y+11 x)=y \pmod{65}$, $9x+33y +33y +121x= y \pmod{65}$,  $y=y \pmod{65}$.
\end{example}

\bigskip

\noindent \footnotesize
{  Victor Shcherbacov\\
 \lq\lq Vladimir Andrunachievici\rq\rq \\
Institute of Mathematics and Computer Science \\
 5 Academiei str., Chi\c{s}in\u{a}u  MD-2028 \\
   Moldova \\
E-mail:   \emph{scerb@math.md }}

\end{document}